\documentclass[11pt,a4paper]{article}

\usepackage[T1]{fontenc}
\usepackage[expansion=false]{microtype}
\usepackage[margin=1in]{geometry}
\usepackage{amsmath,amssymb,amsthm}
\usepackage{mathtools}
\usepackage{graphicx}
\usepackage[dvipsnames]{xcolor}
\usepackage{booktabs}
\usepackage{enumitem}
\usepackage{float}
\usepackage[font=small,labelfont=bf,labelsep=period]{caption}
\usepackage{fancyhdr}
\usepackage[colorlinks=true,linkcolor=MidnightBlue,citecolor=MidnightBlue,urlcolor=MidnightBlue]{hyperref}
\usepackage{cleveref}

\theoremstyle{plain}
\newtheorem{theorem}{Theorem}[section]
\newtheorem{proposition}[theorem]{Proposition}
\newtheorem{corollary}[theorem]{Corollary}

\theoremstyle{definition}
\newtheorem{definition}[theorem]{Definition}

\theoremstyle{plain}
\newtheorem*{theoremrestate}{Theorem 1 (restated)}
\newtheorem*{corollaryrestate}{Corollary A.1}

\newcommand{\R}{\mathbb{R}}
\newcommand{\Z}{\mathbb{Z}}

\newcommand{\one}{\mathbf{1}}
\newcommand{\xneg}{x_{\mathrm{neg}}}
\newcommand{\xpos}{x_{\mathrm{pos}}}
\newcommand{\Sminus}{S_{-}}
\newcommand{\Splus}{S_{+}}
\newcommand{\stepout}{\mathrm{step\_out}}
\newcommand{\worstlb}{\mathrm{worst\_lb}}
\DeclareMathOperator{\IDFT}{IDFT}

\DeclareMathOperator{\spn}{span}

\title{%
  \textbf{On the Failure of Step-Response Tests to Certify\\[4pt]
  Admissibility of Spectral Averaging Operators}}

\author{%
  Justin Grieshop\\[2pt]
  \small Marithmetics \texttt{(public-arch/Marithmetics)}}

\date{}

\begin{document}
\maketitle
\thispagestyle{fancy}

\begin{abstract}
Step responses are widely used as a practical diagnostic for whether a linear
smoothing operator ``behaves'' on bounded signals---e.g., whether it preserves
the unit interval when applied to data normalized to~$[0,1]$.  This paper shows
that step-based diagnostics can be fundamentally misleading for periodic,
translation-invariant operators defined by spectral truncation.  We first give a
sharp characterization: a periodic convolution operator that preserves constants
(unit mass) maps $[0,1]^N$ into $[0,1]^N$ if and only if its convolution kernel
is entrywise nonnegative (equivalently, each row of the associated matrix is a
probability vector).  The proof is constructive and yields worst-case witnesses
in $\{0,1\}^N$ along with certified lower bounds on the magnitude of
boundedness violation.  We then analyze three standard Fourier-domain
``averaging'' constructions---Fej\'er (Ces\`aro) averaging, sharp spectral
truncation, and a signed spectral control---and demonstrate a step-response
blind spot near Nyquist: at the cutoff $K = \mathrm{Nyquist}-1$, the canonical
step can exhibit essentially zero overshoot while the operator remains strongly
non-admissible, with guaranteed outputs below~$0$ and above~$1$ for explicit
binary inputs.  All results are reproduced by a self-contained figure-generation
demo available in the project repository; an archival snapshot of the broader
codebase is available via Zenodo
(\href{https://doi.org/10.5281/zenodo.18689854}{doi:10.5281/zenodo.18689854}).

\end{abstract}

\bigskip
\tableofcontents
\newpage

\section{Introduction}
\label{sec:intro}

Many numerical and signal-processing pipelines apply linear ``averaging''
operators to signals whose range has physical or semantic meaning---probabilities,
normalized intensities, bounded state variables, or concentrations.  In such
settings, a natural requirement is \emph{interval preservation}:

\medskip
\noindent\textbf{Admissibility (informal).}
If an input vector~$x$ satisfies $0 \le x_i \le 1$ for all~$i$, then the
output~$Tx$ should also satisfy $0 \le (Tx)_i \le 1$ for all~$i$.
\medskip

A common informal certification heuristic is to apply the operator to a step
function and check whether the output exhibits undershoot below~$0$ or overshoot
above~$1$.  Step tests are visually compelling, easy to compute, and closely
associated with well-known phenomena such as Gibbs oscillations.

This paper explains why, in and of itself, a good-looking step response does
\emph{not} certify admissibility for spectral averaging operators---even in the
simplest setting of one-dimensional periodic convolution.  The core reason is
structural: interval preservation on~$[0,1]^N$, together with constant
preservation, is equivalent to entrywise nonnegativity of the convolution kernel.
Step tests probe only one input; they can miss negative kernel mass that is
``activated'' by other binary inputs.  Near Nyquist, this failure can be extreme:
the step can be perfectly reconstructed while explicit witnesses force outputs
outside~$[0,1]$ by a fixed amount.

The presentation is intentionally self-contained and elementary.  The goal is not
to introduce new harmonic analysis, but to formalize a certification boundary
that is often treated as folklore in practice: step overshoot is neither a
necessary nor a sufficient certificate for admissibility.

\subsection{Contributions}
\label{sec:contributions}

\begin{enumerate}[label=(\arabic*),leftmargin=2em]
\item \textbf{Characterization theorem (necessary and sufficient).}
  For periodic convolution operators with unit mass, admissibility on~$[0,1]^N$
  is equivalent to kernel nonnegativity (\Cref{thm:main}).

\item \textbf{Constructive witnesses and certified lower bounds.}
  If any kernel coefficient is negative, we construct explicit binary inputs
  $\xneg, \xpos \in \{0,1\}^N$ that force outputs below~$0$ and above~$1$,
  with magnitudes determined by the negative coefficient sum
  (\Cref{cor:violation}).

\item \textbf{Step-response blind spot (``cloaking'' mechanism).}
  For sharp and signed spectral cutoffs at $K = \mathrm{Nyquist}-1$ (with~$N$
  divisible by~$4$), the canonical half-step has zero Nyquist component and
  therefore exhibits essentially zero step violation, even though the kernel has
  substantial negative mass and the operator is provably non-admissible
  (\Cref{sec:blind-spot}).

\item \textbf{Reproducible numeric certificate.}
  We provide a deterministic sweep over~$K$ that separates Fej\'er averaging
  (admissible) from sharp and signed truncations (non-admissible for all
  $K < \mathrm{Nyquist}$), and we visualize the gap between observed step
  violations and guaranteed worst-case violations
  (Figures~\ref{fig:hero}--\ref{fig:obs-vs-guar}).
\end{enumerate}

\begin{figure}[t]
  \centering
  \includegraphics[width=\textwidth]{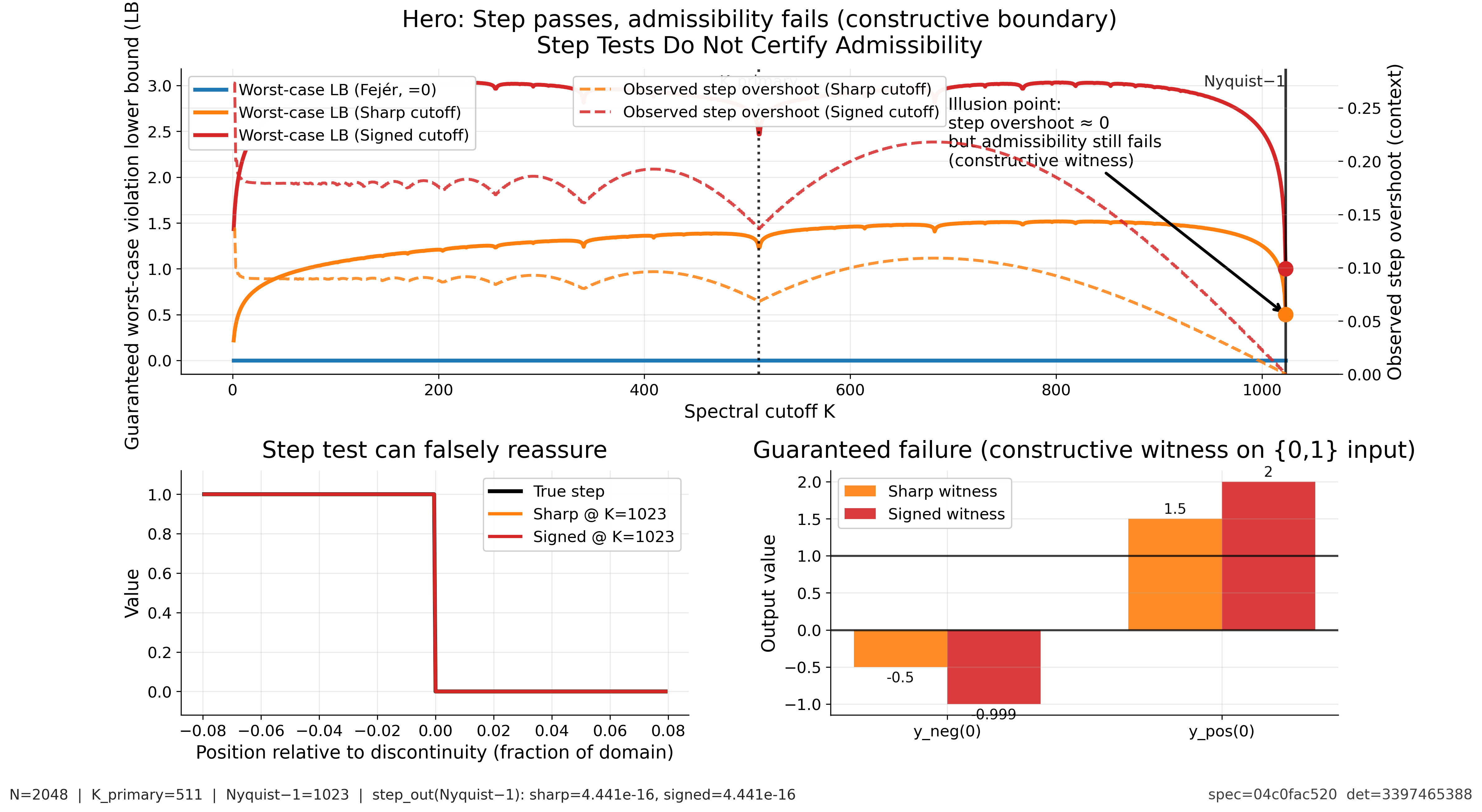}
  \caption{Hero summary.  Observed step-test violation versus spectral cutoff~$K$,
  compared against the certified worst-case lower bound obtained from the
  constructive witness.  The plot highlights a near-Nyquist ``blind spot'' where
  step overshoot vanishes while admissibility fails by a fixed amount.}
  \label{fig:hero}
\end{figure}

\subsection{Scope and standing assumptions}
\label{sec:scope}

We work in the discrete periodic setting~$\Z_N$ with~$N$ even, and focus on
linear time-invariant operators realized as circular convolution with a real
kernel $h \in \R^N$.  The results extend immediately to any row-stochastic
matrix (not necessarily circulant), but the periodic convolution case makes the
spectral constructions transparent.

\section{Setting, Notation, and the Admissibility Requirement}
\label{sec:setting}

Let $N \ge 2$ and identify signals with vectors $x \in \R^N$, indexed by
$n \in \{0,1,\dots,N-1\}$.  For a kernel $h \in \R^N$, define the periodic
convolution operator $T_h : \R^N \to \R^N$ by
\begin{equation}\label{eq:conv}
  (T_h x)[n]
  \;=\;
  \sum_{m=0}^{N-1} h[m]\; x\bigl[n - m \;\;(\mathrm{mod}\; N)\bigr].
\end{equation}
Equivalently, $T_h$ is a circulant matrix whose rows are cyclic shifts of~$h$.

We single out a standard normalization:

\begin{definition}[Unit mass / constant preservation]
\label{def:unit-mass}
The operator~$T_h$ \emph{preserves constants} if $T_h \one = \one$, where
$\one = (1,\dots,1)^\top$.  For convolution this is equivalent to
\[
  \sum_{m=0}^{N-1} h[m] \;=\; 1.
\]
We refer to this condition as the \emph{row-sum condition}.
\end{definition}

\begin{definition}[Admissibility on the unit interval]
\label{def:admissibility}
A linear operator $T : \R^N \to \R^N$ is \emph{admissible} if
\begin{enumerate}[label=(\roman*)]
  \item $T\one = \one$, and
  \item $x \in [0,1]^N \;\Longrightarrow\; Tx \in [0,1]^N$.
\end{enumerate}
\end{definition}

\noindent
The unit interval~$[0,1]$ is chosen for convenience; by affine rescaling the
same discussion applies to any bounded interval~$[a,b]$.

\section{A Characterization of Admissibility for Periodic Convolution}
\label{sec:characterization}

We now state the central theorem.  It isolates what step tests cannot: interval
preservation is a positivity property of the coefficients, not a
frequency-domain heuristic.

\begin{theorem}[Admissibility $\Leftrightarrow$ kernel nonnegativity for periodic convolution]
\label{thm:main}
Let $h \in \R^N$ define a periodic convolution operator~$T_h$.  Assume the
row-sum condition $\sum_{m} h[m] = 1$.  Then the following are equivalent:
\begin{enumerate}[label=(\roman*)]
  \item $T_h$ is admissible on~$[0,1]^N$, i.e.,
        $x \in [0,1]^N \Rightarrow T_h x \in [0,1]^N$.
  \item $h[m] \ge 0$ for all~$m$ (equivalently, every row of the matrix
        for~$T_h$ has nonnegative coefficients).
\end{enumerate}
\end{theorem}

\begin{proof}
\textbf{(ii)~$\Rightarrow$~(i).}\;
Fix~$n$.  If $h[m] \ge 0$ and $\sum_m h[m] = 1$, then
$(T_h x)[n] = \sum_m h[m]\,x[n-m]$ is a convex combination of values
$x[\cdot] \in [0,1]$.  Hence $(T_h x)[n] \in [0,1]$ for every~$n$, so
$T_h x \in [0,1]^N$.

\medskip\noindent
\textbf{(i)~$\Rightarrow$~(ii).}\;
Suppose, to the contrary, that some coefficient is negative.  Because~$T_h$ is
circulant, it suffices to inspect one row (say the row producing output
index~$n = 0$); the same argument applies to any row for general matrices.  Let
$a[j]$ denote the coefficients in that row, so
$(T_h x)[0] = \sum_{j=0}^{N-1} a[j]\,x[j]$ and $\sum_j a[j] = 1$.

Define a binary ``negative-support'' witness $\xneg \in \{0,1\}^N$ by
\[
  \xneg[j] \;=\;
  \begin{cases}
    1, & a[j] < 0,\\
    0, & a[j] \ge 0.
  \end{cases}
\]
Then
\[
  (T_h \xneg)[0]
  \;=\; \sum_{j:\,a[j]<0} a[j]
  \;=:\; \Sminus.
\]
Since at least one $a[j]$ is negative, we have $\Sminus < 0$, hence
$(T_h \xneg)[0] < 0$, contradicting admissibility.

Similarly define $\xpos \in \{0,1\}^N$ by
\[
  \xpos[j] \;=\;
  \begin{cases}
    1, & a[j] > 0,\\
    0, & a[j] \le 0.
  \end{cases}
\]
Then
\[
  (T_h \xpos)[0]
  \;=\; \sum_{j:\,a[j]>0} a[j]
  \;=\; 1 - \Sminus > 1,
\]
again contradicting admissibility.  Therefore no coefficient can be negative, so
$h[m] \ge 0$ for all~$m$.
\end{proof}

\begin{corollary}[Certified worst-case violation from negative mass]
\label{cor:violation}
Under the assumptions of \Cref{thm:main}, define the negative coefficient sum in
a row by
\[
  \Sminus \;:=\; \sum_{j:\,a[j]<0} a[j].
\]
If $\Sminus < 0$, then there exists $x \in \{0,1\}^N$ such that
\[
  \min_n (T_h x)[n] \;\le\; \Sminus
  \qquad\text{and}\qquad
  \max_n (T_h x)[n] \;\ge\; 1 - \Sminus = 1 + |\Sminus|.
\]
Moreover, $\xneg$ and $\xpos$ constructed in the proof achieve these violations
at the corresponding row index.
\end{corollary}

\noindent
This corollary is the basis of our ``guaranteed'' curves in
Figures~\ref{fig:hero} and~\ref{fig:obs-vs-guar}: once $\Sminus < 0$ is
detected, admissibility is not just false---it fails by at least~$|\Sminus|$ in
a worst case that is explicitly constructible.

\begin{figure}[t]
  \centering
  \includegraphics[width=\textwidth]{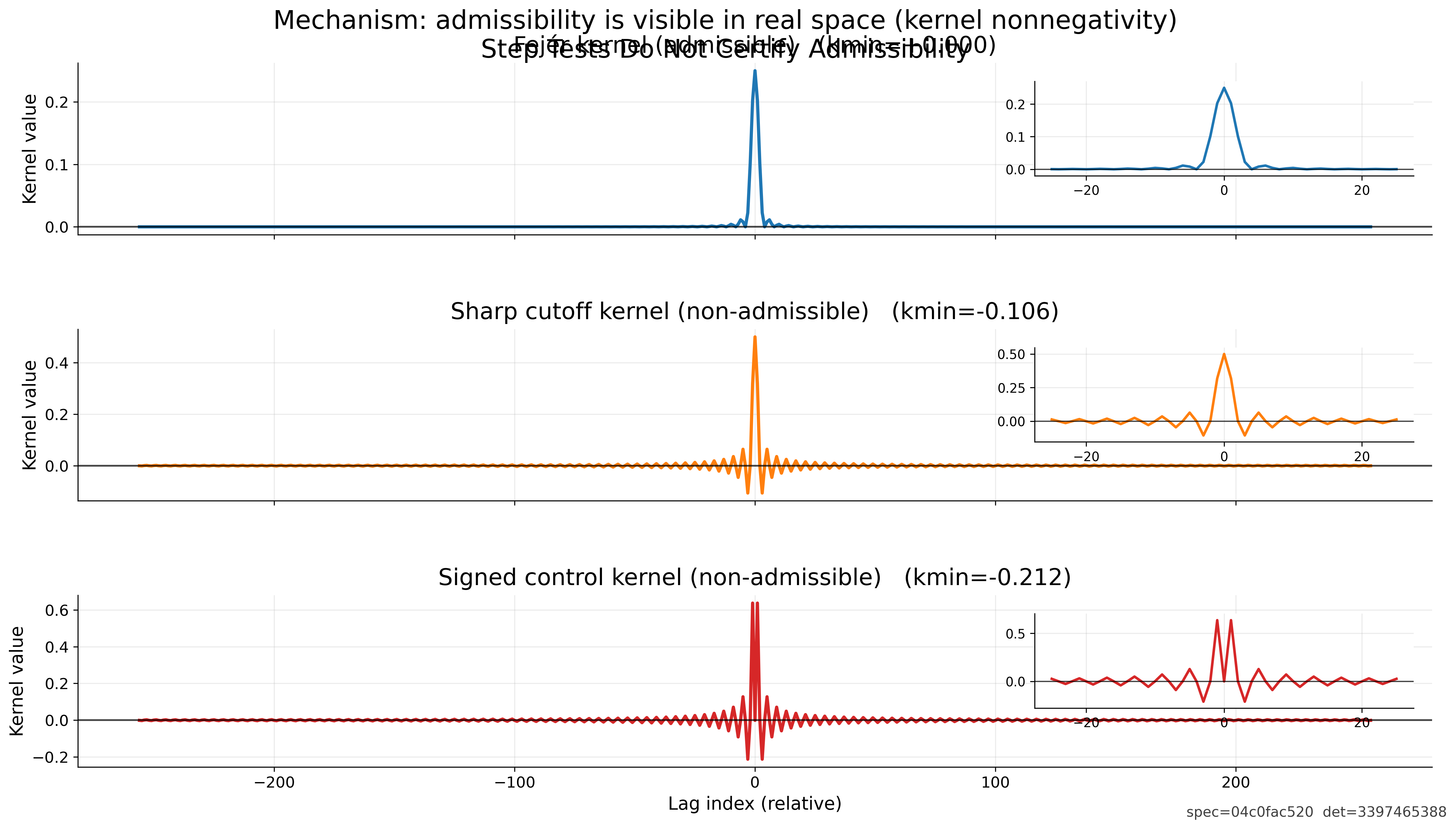}
  \caption{Kernel triptych for three spectral averaging constructions at
  representative cutoffs~$K$.  Fej\'er averaging yields a nonnegative kernel
  (admissible).  Sharp and signed spectral cutoffs produce oscillatory kernels
  with negative coefficients (non-admissible), even when their step responses
  appear visually well-behaved.}
  \label{fig:kernel-triptych}
\end{figure}

\begin{figure}[t]
  \centering
  \includegraphics[width=\textwidth]{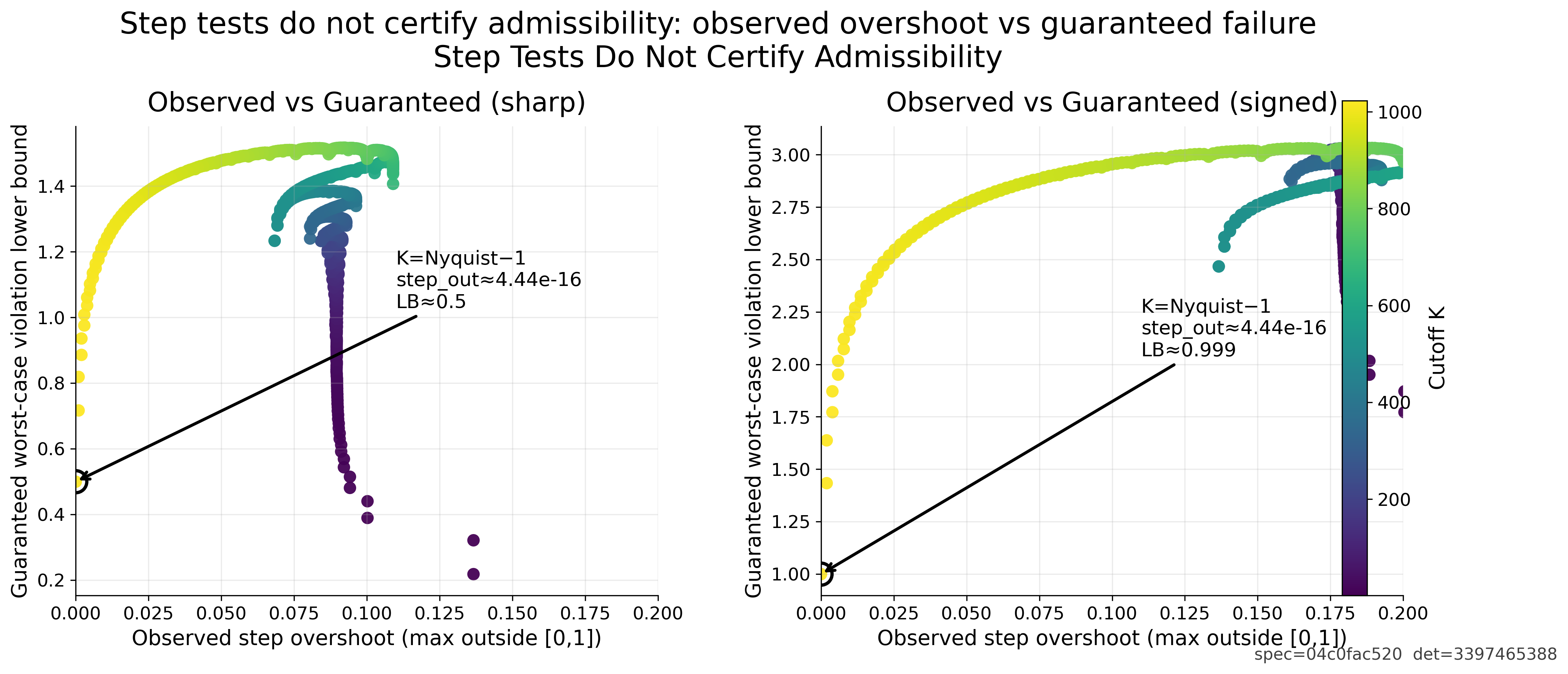}
  \caption{Observed step-test violation versus certified worst-case violation
  across a dense cutoff sweep.  The near-Nyquist region shows a pronounced gap:
  step overshoot can be near zero while the guaranteed violation remains bounded
  away from zero.}
  \label{fig:obs-vs-guar}
\end{figure}

\section{Spectral Averaging Operators and a Step-Response Blind Spot}
\label{sec:spectral}

We now connect \Cref{thm:main} to common Fourier-domain constructions.  Let
$\widehat{x}$ denote the discrete Fourier transform (DFT) of~$x$, with inverse
transform producing a circular convolution kernel~$h$ from a real frequency
response~$H$.  In this setting, a spectral ``averaging'' operator is often
implemented as
\[
  \widehat{y}[k] = H[k]\;\widehat{x}[k],
  \qquad
  y = T_h x,
  \qquad
  h = \IDFT(H),
\]
where $H[0] = 1$ enforces constant preservation $\sum_m h[m] = 1$.

We consider three canonical multipliers, parameterized by a cutoff~$K$ with
$0 \le K \le \mathrm{Nyquist}-1 = \tfrac{N}{2}-1$:

\begin{enumerate}[label=(\arabic*),leftmargin=2em]
  \item \textbf{Sharp truncation:}\;
    $H_{\mathrm{sharp}}[k] = 1$ for $|k| \le K$ and $0$ otherwise.
  \item \textbf{Fej\'er (Ces\`aro) averaging:}\;
    $H_{\mathrm{fejer}}[k] = \max\{0,\; 1 - |k|/(K+1)\}$.
  \item \textbf{Signed spectral control:}\;
    $H_{\mathrm{signed}}[k] = +1$ for $|k| \le K$ and $-1$ otherwise.
\end{enumerate}

\Cref{thm:main} immediately implies: if any resulting kernel has a negative
entry, the operator is not admissible---regardless of what its response to a
particular test signal looks like.  Fej\'er averaging is designed to avoid this
pathology: its kernel is nonnegative, hence admissible.  In contrast, sharp
truncation produces the discrete Dirichlet kernel, which oscillates and assumes
negative values for any nontrivial cutoff.  The signed construction amplifies
this effect.

A particularly instructive case is the ``near Nyquist'' cutoff
$K = \tfrac{N}{2}-1$.  At this cutoff, sharp truncation removes only the
Nyquist bin, and the signed construction differs from the identity only by
flipping the Nyquist bin.  One might expect these operators to behave essentially
like the identity on simple tests.  Indeed, the step response can be essentially
perfect.  However, the kernels remain non-admissible, and the failure is
quantifiable in closed form (\Cref{sec:near-nyquist}).

\subsection{A near-Nyquist blind spot for step-response testing}
\label{sec:near-nyquist}

The most visually persuasive failure mode occurs at the cutoff just below
Nyquist.  In this regime, the step response can look essentially perfect---yet
the operator remains provably non-admissible, with a fixed, certified violation
magnitude.

Throughout this section assume~$N$ is even and let the Nyquist index be
$k_{\mathrm{N}} := N/2$.  Consider the cutoff
\[
  K \;=\; k_{\mathrm{N}} - 1 \;=\; \frac{N}{2} - 1,
\]
so that ``sharp truncation'' removes only the Nyquist bin.

\begin{proposition}[Closed-form kernels at $K = \mathrm{Nyquist}-1$]
\label{prop:closed-form}
Let $T_{\mathrm{sharp}}$ be the sharp spectral truncation operator with
multiplier
\[
  H_{\mathrm{sharp}}[k] \;=\;
  \begin{cases}
    1, & k \neq k_{\mathrm{N}},\\
    0, & k = k_{\mathrm{N}}.
  \end{cases}
\]
Let $T_{\mathrm{signed}}$ be the signed spectral control with multiplier
\[
  H_{\mathrm{signed}}[k] \;=\;
  \begin{cases}
    1, & k \neq k_{\mathrm{N}},\\
    -1, & k = k_{\mathrm{N}}.
  \end{cases}
\]
Then both are periodic convolutions~$T_h$ with kernels
\[
  h_{\mathrm{sharp}}[n] = \delta[n] - \frac{1}{N}(-1)^n,
  \qquad
  h_{\mathrm{signed}}[n] = \delta[n] - \frac{2}{N}(-1)^n,
\]
where $\delta[0] = 1$ and $\delta[n] = 0$ for $n \neq 0$.

Moreover, for each kernel the negative coefficient sum~$\Sminus$ (in any row) is
\[
  \Sminus^{\mathrm{sharp}}
  = -\frac{(N/2)-1}{N}
  = -\Bigl(\frac{1}{2} - \frac{1}{N}\Bigr),
  \qquad
  \Sminus^{\mathrm{signed}}
  = -\frac{2\bigl((N/2)-1\bigr)}{N}
  = -\Bigl(1 - \frac{2}{N}\Bigr).
\]
\end{proposition}

\begin{proof}
The identity operator has multiplier $H_{\mathrm{id}}[k] \equiv 1$, whose
inverse DFT is~$\delta[n]$.  The difference between~$H_{\mathrm{id}}$ and
$H_{\mathrm{sharp}}$ is the indicator at the Nyquist bin, so
$h_{\mathrm{sharp}} = \delta - \IDFT(\mathbf{1}_{\{k_{\mathrm{N}}\}})$.  The
inverse DFT of a unit impulse at~$k_{\mathrm{N}}$ is $\frac{1}{N}(-1)^n$,
giving the first formula.  The signed construction flips the Nyquist bin from
$+1$ to~$-1$, i.e., subtracts $2\mathbf{1}_{\{k_{\mathrm{N}}\}}$, yielding
$h_{\mathrm{signed}} = \delta - \frac{2}{N}(-1)^n$.

For~$h_{\mathrm{sharp}}$, all even indices $n \neq 0$ have coefficient~$-1/N$,
while odd indices have~$+1/N$, and $h_{\mathrm{sharp}}[0] = 1 - 1/N > 0$.
Thus the negative coefficients occur exactly at the $(N/2)-1$ even indices
$n \in \{2,4,\dots,N-2\}$, and their sum is $-((N/2)-1)/N$.  The signed case is
identical with~$-2/N$ on those even indices.
\end{proof}

\begin{corollary}[Guaranteed violation magnitude at $K = \mathrm{Nyquist}-1$]
\label{cor:nyquist}
Both operators are non-admissible for every even $N \ge 4$.  In particular, by
\Cref{cor:violation} there exist binary inputs
$\xneg, \xpos \in \{0,1\}^N$ such that, at some output index,
\[
  (T_{\mathrm{sharp}}\xneg) < 0,
  \quad
  (T_{\mathrm{sharp}}\xpos) > 1,
  \qquad
  (T_{\mathrm{signed}}\xneg) < 0,
  \quad
  (T_{\mathrm{signed}}\xpos) > 1,
\]
with certified bounds
\[
  \min(T_{\mathrm{sharp}}x) \le -\Bigl(\tfrac{1}{2}-\tfrac{1}{N}\Bigr),
  \quad
  \max(T_{\mathrm{sharp}}x) \ge 1+\Bigl(\tfrac{1}{2}-\tfrac{1}{N}\Bigr),
\]
and
\[
  \min(T_{\mathrm{signed}}x) \le -\Bigl(1-\tfrac{2}{N}\Bigr),
  \quad
  \max(T_{\mathrm{signed}}x) \ge 2-\tfrac{2}{N}.
\]
These bounds are not asymptotic; they hold exactly for the constructed witnesses.
\end{corollary}

\subsubsection{Why the step test can look ``perfect'' anyway}
\label{sec:blind-spot}

Define the canonical half-step (periodic) input
\[
  x_{\mathrm{step}}[n] \;=\;
  \begin{cases}
    0, & 0 \le n \le \tfrac{N}{2}-1,\\
    1, & \tfrac{N}{2} \le n \le N-1.
  \end{cases}
\]
The Nyquist Fourier mode in the time domain is the alternating sequence~$(-1)^n$.
The Nyquist coefficient of~$x_{\mathrm{step}}$ is proportional to the inner
product $\sum_{n=0}^{N-1} x_{\mathrm{step}}[n](-1)^n$.  Because
$x_{\mathrm{step}}$ is supported on exactly~$N/2$ consecutive indices, and
because those~$N/2$ indices contain equally many even and odd integers
when~$N/2$ is even (as in the standard power-of-two setting), we have the exact
cancellation
\[
  \sum_{n=0}^{N-1} x_{\mathrm{step}}[n](-1)^n
  \;=\;
  \sum_{n=N/2}^{N-1} (-1)^n
  \;=\; 0.
\]
Consequently, removing or flipping only the Nyquist bin has no effect on
$x_{\mathrm{step}}$:
\[
  T_{\mathrm{sharp}}\,x_{\mathrm{step}} = x_{\mathrm{step}},
  \qquad
  T_{\mathrm{signed}}\,x_{\mathrm{step}} = x_{\mathrm{step}},
\]
up to roundoff in floating arithmetic.  This creates a \emph{step-response blind
spot} (a ``cloaking effect''): the input fails to excite precisely the frequency
component where these operators differ from the identity, even though their
real-space kernels contain substantial negative mass.  \Cref{fig:hero} highlights
this phenomenon at $K = \mathrm{Nyquist}-1$, where the observed step violation
can be essentially zero while the guaranteed worst-case violation remains
near~$1/2$ (sharp) or near~$1$ (signed).

The conclusion is conceptual and general: step-response testing probes a single
direction in~$[0,1]^N$.  Admissibility, by contrast, is a uniform guarantee over
the entire hypercube, and \Cref{thm:main} shows it is governed by
coefficient-wise positivity---something a single input cannot certify.

\subsection{Quantitative separation between step overshoot and guaranteed failure}
\label{sec:quantitative}

The preceding section explains why step-based tests can miss violations: the step
does not necessarily excite the parts of the kernel responsible for
inadmissibility.  We now quantify this mismatch by sweeping the spectral cutoff
and comparing two diagnostics:

\begin{enumerate}[label=\arabic*.,leftmargin=2em]
  \item \textbf{Observed step overshoot} (a step-response diagnostic).
    Let $x_{\mathrm{step}}$ be the $0$--$1$ step input on the discrete periodic
    grid.  For an operator~$T$ with impulse response~$h$, let
    $y_{\mathrm{step}} = T x_{\mathrm{step}}$.  Define
    \[
      \stepout(T)
      \;:=\;
      \max\bigl\{0,\;\;-\!\min_n y_{\mathrm{step}}[n],\;\;
        \max_n y_{\mathrm{step}}[n] - 1\bigr\}.
    \]
    Thus $\stepout(T) = 0$ means the step response stays within~$[0,1]$
    everywhere on the grid.

  \item \textbf{Guaranteed violation lower bound} (a certificate implied by
    kernel negativity).  As established in \Cref{cor:violation}, for any
    mass-preserving convolution operator ($\sum_n h[n] = 1$),
    \[
      \mathrm{LB}(T) \;:=\; |\Sminus(h)|
    \]
    is a certified lower bound on worst-case violation: there exists an input
    $x \in \{0,1\}^N$ such that the output leaves~$[0,1]$ by at
    least~$\mathrm{LB}(T)$.
\end{enumerate}

We compute $(\stepout, \mathrm{LB})$ across the full cutoff sweep
$K = 1, 2, \dots, \mathrm{Nyquist}-1$ for three families: Fej\'er (admissible),
sharp cutoff, and signed control.  The result is striking: there is no reliable
relationship between~$\stepout$ and~$\mathrm{LB}$.  In particular, there are
cutoffs where~$\stepout$ is essentially machine zero
while~$\mathrm{LB}$ remains order-one.

\Cref{fig:obs-vs-guar} makes the certification gap visually explicit.  For the
admissible Fej\'er family, the step test and the certificate agree:
$\stepout = 0$ and $\mathrm{LB} = 0$ for all~$K$.  By contrast, both
non-admissible families exhibit broad regions of false reassurance, culminating
in an extreme point at the largest cutoff:

\begin{itemize}[leftmargin=2em]
  \item \textbf{Sharp cutoff at $K = \mathrm{Nyquist}-1$:}\;
    $\stepout \approx 4.44 \times 10^{-16}$ (numerically indistinguishable
    from~$0$), yet $\mathrm{LB} \approx 0.4995$.
  \item \textbf{Signed control at $K = \mathrm{Nyquist}-1$:}\;
    $\stepout \approx 4.44 \times 10^{-16}$, yet
    $\mathrm{LB} \approx 0.9990$.
\end{itemize}

In words: the step response can look perfect while a constructive $\{0,1\}$
witness is guaranteed to violate the admissibility bounds by~${\sim}0.5$--$1.0$.
The step test is therefore not a sound certification method for admissibility,
even for highly structured, shift-invariant spectral operators.

This cutoff sweep also clarifies that the phenomenon is not confined to a single
tuned parameter value.  The mismatch persists across wide ranges of~$K$: observed
overshoot can remain small (or even decrease) while the guaranteed violation
lower bound remains substantial.  The step response is therefore an unreliable
proxy for the operator's order-preserving behavior on the full $0$--$1$ cube.

\subsection{Mechanism: kernel nonnegativity and a step-test blind spot}
\label{sec:mechanism}

The certificate in \Cref{fig:hero} is intentionally paradoxical: the observed
step response can look essentially perfect while admissibility fails by a fixed,
constructive margin.  This is not a numerical artifact.  It is a structural
mismatch between what a step test probes and what admissibility requires.

\subsubsection{What a step test actually measures}

Let~$H$ be a translation-invariant discrete operator on~$\Z_N$ represented by
circular convolution with an impulse response (kernel)~$h$, i.e.,
\[
  (Hx)[n]
  \;=\;
  \sum_{m \in \Z_N} h[n-m]\,x[m].
\]
Assume the row-sum (mass preservation) condition $\sum_{m \in \Z_N} h[m] = 1$,
so that constant inputs are fixed: $H\one = \one$.  This condition is standard in
filtering and is essential for interpreting any ``step passes'' claim: if
constants are not preserved, a step test can fail for trivial reasons unrelated
to admissibility.

A step test typically evaluates~$Hx$ on a Heaviside-type input
$s \in \{0,1\}^N$ that is constant on each side of a single discontinuity.  For
definiteness, take the circular step
\[
  s[m] \;=\;
  \begin{cases}
    1, & m \in \{0,1,\dots,N/2-1\},\\
    0, & m \in \{N/2,\dots,N-1\}.
  \end{cases}
\]
Then the output at a point near the discontinuity is a partial sum of the kernel:
\[
  (Hs)[0] \;=\; \sum_{m=0}^{N/2-1} h[-m].
\]
Thus, the step test is not measuring whether~$h$ is nonnegative; it is measuring
whether certain contiguous cumulative sums of~$h$ stay within~$[0,1]$.  In other
words, a step test probes a specific averaging functional of the kernel, not the
full sign structure that governs admissibility on~$[0,1]^N$.

\subsubsection{What admissibility requires (and why negativity is fatal)}

Admissibility, as used here, is the requirement that~$H$ maps~$[0,1]^N$
to~$[0,1]^N$.  For convolution operators with $\sum h = 1$, this is equivalent
to kernel nonnegativity: $h[m] \ge 0$ for all~$m$.  The necessity direction is
elementary and is the key to the constructive witness in \Cref{fig:hero}:

\begin{itemize}[leftmargin=2em]
  \item If there exists an index~$m_0$ with $h[m_0] < 0$, consider an input
    $x \in \{0,1\}^N$ that places ones exactly on the negative coefficients of
    the relevant row (and zeros elsewhere).  Then $(Hx)[n]$ becomes a sum of
    negative terms and is strictly negative at that location.  This produces an
    explicit $[0,1]$-violating output.
\end{itemize}

Because this witness uses the support of the negative part of the kernel, it can
remain large even when the step response is nearly ideal.

\subsubsection{The blind spot: cancellation of high-frequency negative lobes
  under step averaging}

\Cref{fig:kernel-triptych} makes the mechanism visible in real space.  The
Fej\'er kernel is nonnegative everywhere (up to numerical noise at machine
precision), while the sharp cutoff and signed-control kernels exhibit oscillatory
sidelobes with substantial negative mass.  The insets in
\Cref{fig:kernel-triptych} emphasize a critical qualitative distinction:

\begin{itemize}[leftmargin=2em]
  \item \textbf{Fej\'er (admissible):} no negative lobes; the kernel is a
    genuine averaging mask.
  \item \textbf{Sharp cutoff (non-admissible):} alternating sidelobes with a
    negative minimum; the operator is not order-preserving.
  \item \textbf{Signed control (non-admissible):} deeper negative lobes; the
    violation margin is larger.
\end{itemize}

Why, then, can the step overshoot be essentially zero at the Nyquist-adjacent
cutoff while the operator is still non-admissible by a fixed amount?

The reason is that a step input aggregates kernel coefficients in a contiguous
block (a half-domain sum in the canonical step above).  For oscillatory kernels
whose negative coefficients occur in rapidly alternating patterns, contiguous
summation produces strong cancellation: negative lobes are ``paired'' with nearby
positive lobes and their contributions nearly cancel in the cumulative sum that
the step test measures.  This cancellation becomes more pronounced as the
oscillations become higher frequency, which is exactly what happens as the
spectral cutoff approaches Nyquist.

This creates a practical blind spot for step-based certification: the step input
can fail to excite (or can inadvertently average away) the negative part of the
kernel that drives inadmissibility on other $\{0,1\}$ inputs.

\Cref{fig:hero} isolates the resulting ``illusion point.''  Near
$K = \mathrm{Nyquist}-1$, the observed step overshoot is~$\approx 0$, suggesting
(incorrectly) that the operator behaves like a benign approximation to the
identity.  Yet the constructive witness (a $\{0,1\}$ input aligned with the
kernel's negative part) produces a guaranteed violation bounded away from zero.
The gap is not a calibration issue; it reflects the fact that step tests are
testing the wrong functional.

\subsubsection{Connection to classical Fourier summation: Dirichlet versus
  Fej\'er}

The triptych in \Cref{fig:kernel-triptych} is also consistent with a classical
harmonic-analysis distinction.  Sharp spectral truncation corresponds to
Dirichlet-type kernels, which are oscillatory and sign-indefinite; Fej\'er
summation corresponds to Ces\`aro averaging, whose kernel is nonnegative and
therefore defines an averaging operator.  From the admissibility perspective,
this classical positivity property is decisive: Fej\'er-type averaging is
compatible with $[0,1]$-preservation, while sharp truncation generically is not.

This connection matters for interpretation: the counterexample is not a contrived
corner case.  It reflects a structural incompatibility between sign-indefinite
kernels and order-preserving behavior, and it explains why ``good-looking'' step
responses can coexist with rigorous admissibility failure.

\subsection{Certified violations across the cutoff sweep}
\label{sec:certified-sweep}

The witness construction in \Cref{sec:characterization} turns admissibility into
a computable certificate: for any constant-preserving convolution operator~$T_h$
(i.e., $\sum_m h[m] = 1$), a single scalar
\[
  \Sminus(h)
  \;:=\;
  \sum_{m:\,h[m]<0} h[m]
\]
quantifies a guaranteed worst-case failure whenever $\Sminus(h) < 0$.
Specifically, define the binary witness signals
\[
  \xneg[m] := \mathbf{1}\{h[-m]<0\},
  \qquad
  \xpos[m] := \mathbf{1}\{h[-m]>0\}.
\]
Then at the corresponding output index,
\[
  (T_h \xneg)[0] = \Sminus(h) < 0,
  \qquad
  (T_h \xpos)[0] = 1 - \Sminus(h) = 1 + |\Sminus(h)| > 1.
\]
Thus, for any mass-preserving operator,
\[
  \Sminus(h) < 0
  \quad\Longrightarrow\quad
  \exists\,x \in \{0,1\}^N:\;
  \min(T_h x) \le \Sminus(h)
  \;\;\text{and}\;\;
  \max(T_h x) \ge 1 + |\Sminus(h)|.
\]
This yields a certified lower bound on violation magnitude:
\[
  \worstlb(h) \;:=\; -\Sminus(h) = |\Sminus(h)|,
\]
which is independent of any particular ``test signal'' and depends only on the
kernel.

We now contrast this certificate with step testing across a dense sweep of
spectral cutoffs.  For each cutoff $K \in \{1,\dots,K_{\mathrm{Nyq}}\}$, we
compute:

\begin{itemize}[leftmargin=2em]
  \item Negative mass~$\Sminus(h_K)$ and its certified bound
    $\worstlb(h_K) = |\Sminus(h_K)|$.
  \item Observed step violation~$\stepout(K)$, defined as the maximum amount by
    which the step response leaves~$[0,1]$:
    \[
      \stepout(K)
      \;:=\;
      \max\Bigl\{0,\;\;
        -\min_n y_{\mathrm{step}}[n],\;\;
        \max_n y_{\mathrm{step}}[n] - 1\Bigr\},
      \qquad
      y_{\mathrm{step}} := T_{h_K} x_{\mathrm{step}}.
    \]
\end{itemize}

The results, summarized in \Cref{fig:obs-vs-guar}, cleanly separate the
admissible and non-admissible families:

\begin{itemize}[leftmargin=2em]
  \item \textbf{Fej\'er averaging} has $h_K[m] \ge 0$ for all~$m$ and all~$K$,
    hence $\Sminus(h_K) = 0$, $\worstlb(h_K) = 0$, and $\stepout(K) = 0$
    throughout the sweep.  This is the expected behavior of a bona fide averaging
    operator.
  \item \textbf{Sharp spectral truncation} exhibits kernel negativity for every
    $K < K_{\mathrm{Nyq}}$, so $\Sminus(h_K) < 0$ and $\worstlb(h_K) > 0$ for
    every such cutoff.  Importantly, $\stepout(K)$ is not monotone
    in~$\worstlb(h_K)$: near the Nyquist limit the step test can report
    essentially no overshoot even while the guaranteed worst-case violation
    remains substantial.
  \item \textbf{Signed control} behaves similarly but with larger negative mass;
    the certified bound~$\worstlb(h_K)$ is correspondingly larger, even when the
    step response appears benign.
\end{itemize}

A particularly instructive comparison is between a mid-band cutoff $K = 511$ and
the near-Nyquist cutoff $K = 1023$ (with $N = 2048$, so
$K_{\mathrm{Nyq}} = 1023$).  At $K = 511$, both sharp and signed truncations
already show visible step overshoot, and the certificates indicate much larger
possible failures:
\[
  \stepout(511) \approx 6.8 \times 10^{-2}\;\text{(sharp)},
  \qquad
  \worstlb(h_{511}) \approx 1.233\;\text{(sharp)},
\]
\[
  \stepout(511) \approx 1.37 \times 10^{-1}\;\text{(signed)},
  \qquad
  \worstlb(h_{511}) \approx 2.467\;\text{(signed)}.
\]
At $K = 1023$, however, the step test becomes maximally misleading:
$\stepout(1023)$ drops to floating-point noise, while the certified failures
persist:
\[
  \stepout(1023) \approx 4.4 \times 10^{-16}
  \quad\text{but}\quad
  \worstlb(h_{1023}) \approx 0.4995\;\text{(sharp)},
  \quad
  0.9990\;\text{(signed)}.
\]
In other words, even when the step response appears perfectly bounded, sharp
truncation still admits a binary input that forces an output below~$-0.4995$ and
another that forces an output above~$1.4995$; the signed control admits failures
below~$-0.9990$ and above~$1.9990$.  These are not numerical artifacts: they are
algebraic consequences of negative coefficients under the mass constraint
$\sum_m h[m] = 1$.

\subsection{Why step tests can pass while admissibility fails: a near-Nyquist
  blind spot}
\label{sec:why-step-fails}

The preceding section shows a numerical fact that can be surprising on first
contact: for near-Nyquist cutoffs, the step response of sharp (and signed)
spectral truncation can be essentially perfect---$\stepout(K) \approx 0$---even
though the operator is provably non-admissible and admits large certified
violations (\Cref{fig:hero,fig:obs-vs-guar}).  This is not a contradiction.  It
is a limitation of what the step test actually probes.

To make the distinction explicit, recall that for any convolution kernel~$h$, the
output at a fixed index is a dot product with the input:
\[
  (T_h x)[0]
  \;=\;
  \sum_{m=0}^{N-1} h[m]\; x[-m].
\]
A step test evaluates this dot product for one particular structured
input~$x_{\mathrm{step}}$, typically an indicator of a contiguous half-interval
(up to periodic convention).  Consequently, for each output index, the step
response aggregates the kernel coefficients through contiguous block sums.  For
example, if $x_{\mathrm{step}}$ equals~$1$ on a contiguous set
$I \subset \{0,\dots,N-1\}$ and~$0$ elsewhere, then
\[
  (T_h x_{\mathrm{step}})[0]
  \;=\;
  \sum_{m:\,-m \in I} h[m],
\]
which is the sum of coefficients over a contiguous block of indices determined by
the step location.  In short:

\begin{itemize}[leftmargin=2em]
  \item The step test probes only interval (block) sums of the kernel.
  \item Admissibility requires all coefficients to be nonnegative, equivalently
    that every subset-sum induced by $\{0,1\}$-valued inputs lies in~$[0,1]$.
\end{itemize}

The witness construction in \Cref{sec:characterization} exploits exactly this
gap.  If any coefficient is negative, one can choose a binary input~$\xneg$ that
places~$1$ precisely on the negative support of the relevant row, forcing the
output below~$0$.  A contiguous block indicator (a step) generally cannot realize
that sign-aligned selection pattern, especially when the negative coefficients
oscillate rapidly.

This mismatch becomes most pronounced near Nyquist for sharp truncation.  As
shown in the kernel triptych (\Cref{fig:kernel-triptych}), the sharp-truncation
kernel has alternating-sign sidelobes that become increasingly high-frequency as
the cutoff approaches Nyquist.  When such an oscillatory kernel is summed over
long contiguous blocks---as a step test effectively does---positive and negative
contributions can cancel to produce a monotone-looking transition with little or
no overshoot.  Yet the coefficients remain sign-indefinite, and the operator
remains non-admissible by \Cref{thm:main}.

We refer to this as a \emph{near-Nyquist blind spot} of step testing: a
structured input whose block-sum averaging can ``mask'' oscillatory negativity in
the impulse response, even when that negativity is large enough to yield
substantial worst-case violations on other binary inputs.  In \Cref{fig:hero},
this blind spot appears as a visually compelling ``all clear'' step response at
$K = 1023$ despite the certified outputs $y_{\mathrm{neg}}(0) < 0$ and
$y_{\mathrm{pos}}(0) > 1$.

This mechanism clarifies the logical status of step testing in the present
setting.  Because admissibility requires boundedness for all $x \in [0,1]^N$, a
step test is a necessary sanity check: every admissible operator must pass it.
But it is not sufficient.  The correct certification criterion is the one given
by \Cref{thm:main}---nonnegativity of the kernel coefficients (equivalently,
nonnegativity of every row of the operator matrix)---and the witness-based
quantity~$\Sminus(h)$ provides an immediate quantitative certificate whenever the
criterion is violated.

\section{Certified Sweep over Spectral Cutoffs and a Nyquist-Blind
  Counterexample}
\label{sec:sweep}

This section instantiates the general criterion above for three widely used
periodic spectral constructions and quantifies---over all cutoffs---the gap
between an observed step response and a certified worst-case violation.  The main
message is that the ``step overshoot'' diagnostic can exhibit a blind spot: there
are cutoffs at which the canonical step response is essentially exact, even
though the operator is provably non-admissible by a fixed margin on a simple
$\{0,1\}$-valued witness input.

\subsection{Operators and diagnostics}
\label{sec:ops-diag}

We work on a periodic grid of length~$N$ (even), with indices
$n \in \{0,\dots,N-1\}$, and use the unitary discrete Fourier transform.  For
each cutoff $K \in \{1,\dots,N/2-1\}$ we define a Fourier multiplier~$H_K(k)$
on integer frequencies~$k$ and the associated linear operator
\[
  T_K x
  \;=\;
  \mathcal{F}^{-1}\!\bigl(H_K \cdot \widehat{x}\bigr),
\]
equivalently a circular convolution $T_K x = h_K * x$ with kernel
$h_K = \mathcal{F}^{-1}(H_K)$.

We evaluate three multipliers:
\begin{enumerate}[label=\arabic*.,leftmargin=2em]
  \item \textbf{Fej\'er (Ces\`aro) low-pass.}
    \[
      H^{\mathrm{F}}_K(k) \;=\;
      \begin{cases}
        1 - \frac{|k|}{K+1}, & |k| \le K,\\
        0, & |k| > K.
      \end{cases}
    \]
  \item \textbf{Sharp low-pass (Dirichlet truncation).}
    \[
      H^{\mathrm{Sh}}_K(k) \;=\;
      \begin{cases}
        1, & |k| \le K,\\
        0, & |k| > K.
      \end{cases}
    \]
  \item \textbf{Signed control (low-pass preserved, complement flipped).}
    \[
      H^{\pm}_K(k) \;=\;
      \begin{cases}
        +1, & |k| \le K,\\
        -1, & |k| > K.
      \end{cases}
    \]
\end{enumerate}

All three are normalized at DC: $H(0) = 1$.  In our implementation we
additionally verify the row-sum (mass preservation) condition
$\sum_{n=0}^{N-1} h_K[n] = 1$ to machine precision for every~$K$.  This
condition is important: once $\sum h_K = 1$, any negative mass must be
compensated by positive mass, which forces a corresponding upper-violation
witness (\Cref{sec:characterization}).

We track three diagnostics:

\begin{itemize}[leftmargin=2em]
  \item \textbf{Observed step overshoot.}  Let $s \in [0,1]^N$ be the canonical
    midpoint step,
    \[
      s[n] \;=\;
      \begin{cases}
        1, & n < N/2,\\
        0, & n \ge N/2,
      \end{cases}
    \]
    and let $y = T_K s$.  We define
    \[
      \stepout(K)
      \;=\;
      \max\Bigl\{\max_n (y[n]-1)_+,\;\;\max_n (-y[n])_+\Bigr\},
    \]
    i.e., the maximum excursion of the step response outside~$[0,1]$.

  \item \textbf{Kernel negativity and certified worst-case violation.}  Define
    the negative and positive sums
    \[
      \Sminus(h_K) = \sum_{n:\,h_K[n]<0} h_K[n],
      \qquad
      \Splus(h_K) = \sum_{n:\,h_K[n]>0} h_K[n].
    \]
    Under $\sum h_K = 1$, we have $\Splus(h_K) = 1 - \Sminus(h_K)$.
    \Cref{thm:main} implies the certified lower bound
    \[
      \mathrm{LB}(K)
      \;=\;
      \max\{-\Sminus(h_K),\;\Splus(h_K) - 1\}
      \;=\;
      -\Sminus(h_K),
    \]
    and provides a constructive $\{0,1\}$-valued witness achieving this
    violation at a specific coordinate.

  \item \textbf{Minimum kernel coefficient.}  We record
    $k_{\min}(K) = \min_n h_K[n]$ as a compact ``real-space mechanism''
    indicator.
\end{itemize}

\subsection{Full $K$-sweep: what the step test sees vs.\ what it cannot certify}
\label{sec:full-sweep}

We fix $N = 2048$ and sweep all cutoffs $K = 1, 2, \dots, N/2-1 = 1023$.  For
each~$K$ and each operator family we compute $\stepout(K)$, $\Sminus(h_K)$,
and~$\mathrm{LB}(K)$.

The outcomes align with the structural theorem and make the ``false reassurance''
concrete:

\begin{itemize}[leftmargin=2em]
  \item \textbf{Fej\'er is admissible across the sweep.}  We observe
    $k_{\min}(K) = 0$ and $\Sminus(h_K) = 0$ (hence $\mathrm{LB}(K) = 0$)
    for all~$K$ in the sweep, consistent with nonnegativity of the Fej\'er
    kernel.  The step response stays within~$[0,1]$, and $\stepout(K) = 0$.

  \item \textbf{Sharp and signed are non-admissible across the sweep.}  For
    every cutoff tested, both the sharp and signed kernels exhibit negative
    coefficients ($k_{\min}(K) < 0$), hence $\Sminus(h_K) < 0$ and
    $\mathrm{LB}(K) > 0$.  This yields a certified, input-independent statement:
    there exists a $\{0,1\}$-valued input for which the output violates
    admissibility by at least~$\mathrm{LB}(K)$ in magnitude at a specific
    coordinate.

  \item \textbf{Yet the observed step overshoot can be small---and at one
    cutoff, essentially zero.}  As~$K$ increases, the observed step
    overshoot~$\stepout(K)$ decreases, and at $K = N/2-1$ it drops to
    machine-level ($\approx 4.4 \times 10^{-16}$) for both sharp and signed
    operators.  At exactly the same cutoff, the certified worst-case violation
    lower bound remains macroscopically positive:
    $\mathrm{LB}(1023) \approx 0.5$ (sharp) and
    $\mathrm{LB}(1023) \approx 1$ (signed).
\end{itemize}

\Cref{fig:obs-vs-guar} summarizes the phenomenon in a single view: for a wide
range of cutoffs there is no monotone relationship between the observed step
overshoot (horizontal axis) and the certified worst-case violation lower bound
(vertical axis).  In particular, the circled ``Nyquist$-$1'' point lies near the
horizontal axis (step looks perfect) while still high on the vertical axis
(operator is provably non-admissible).

\subsection{The Nyquist blind spot: an explicit, closed-form explanation
  at $K = N/2-1$}
\label{sec:nyquist-closed}

The special cutoff $K = N/2-1$ admits a simple analytic description that explains
why the step test becomes uninformative.

Let $k_{\mathrm{Nyq}} = N/2$ denote the Nyquist frequency on an even grid,
whose real-valued Fourier mode is~$(-1)^n$.  Consider the sharp operator at
$K = N/2-1$.  Its multiplier is ``all-pass except Nyquist'':
\[
  H^{\mathrm{Sh}}_{N/2-1}(k) \;=\;
  \begin{cases}
    1, & k \neq k_{\mathrm{Nyq}},\\
    0, & k = k_{\mathrm{Nyq}}.
  \end{cases}
\]
Therefore,
\[
  T^{\mathrm{Sh}}_{N/2-1} \;=\; I \;-\; P_{\mathrm{Nyq}},
\]
where~$P_{\mathrm{Nyq}}$ is the orthogonal projector onto the Nyquist mode.
The corresponding convolution kernel is explicit:
\[
  h^{\mathrm{Sh}}_{N/2-1}[n]
  \;=\;
  \delta[n] \;-\; \frac{1}{N}(-1)^n.
\]
From this expression we can compute the negativity certificate exactly.  Among
the~$N$ indices, exactly $N/2-1$ nonzero even lags have coefficient~$-1/N$,
hence
\[
  \Sminus\!\bigl(h^{\mathrm{Sh}}_{N/2-1}\bigr)
  \;=\;
  -\frac{N/2-1}{N}
  \;=\;
  -\Bigl(\frac{1}{2} - \frac{1}{N}\Bigr),
  \qquad
  \mathrm{LB}(N/2-1) = \frac{1}{2} - \frac{1}{N}.
\]
For $N = 2048$, this yields $\mathrm{LB}(1023) = 0.49951171875$, matching the
certified value reported in \Cref{fig:hero}.

Now examine the step test.  The midpoint step $s[n] = \mathbf{1}_{n<N/2}$ has
zero Nyquist coefficient when~$N/2$ is even (as in our case $N = 2048$):
\[
  \langle s, (-1)^n \rangle
  \;=\;
  \sum_{n=0}^{N/2-1} (-1)^n
  \;=\; 0.
\]
Consequently $P_{\mathrm{Nyq}} s = 0$ and
\[
  T^{\mathrm{Sh}}_{N/2-1}\, s
  \;=\;
  \bigl(I - P_{\mathrm{Nyq}}\bigr) s
  \;=\; s.
\]
That is, at $K = N/2-1$ the sharp operator reproduces the canonical step
\emph{exactly}, so $\stepout(N/2-1) = 0$, even though
$\mathrm{LB}(N/2-1) \approx 1/2$ certifies a substantial admissibility failure
on other inputs.  This is the step test's blind spot in its simplest form: the
tested signal happens to be orthogonal to the single mode on which the operator
differs from the identity.

The signed operator at the same cutoff is equally transparent.  At $K = N/2-1$,
\[
  H^{\pm}_{N/2-1}(k) \;=\;
  \begin{cases}
    1, & k \neq k_{\mathrm{Nyq}},\\
    -1, & k = k_{\mathrm{Nyq}},
  \end{cases}
  \qquad\Longrightarrow\qquad
  T^{\pm}_{N/2-1} = I - 2P_{\mathrm{Nyq}}.
\]
Hence its kernel is
$h^{\pm}_{N/2-1}[n] = \delta[n] - \frac{2}{N}(-1)^n,$
and the negativity sum is
\[
  \Sminus\!\bigl(h^{\pm}_{N/2-1}\bigr)
  \;=\;
  -\frac{2(N/2-1)}{N}
  \;=\;
  -(1 - \tfrac{2}{N}),
  \qquad
  \mathrm{LB}(N/2-1) = 1 - \frac{2}{N}.
\]
For $N = 2048$, $\mathrm{LB}(1023) = 0.9990234375$.  Yet the midpoint step
again satisfies $P_{\mathrm{Nyq}} s = 0$, so $T^{\pm}_{N/2-1} s = s$ and the
step test again reports essentially perfect behavior.

This closed-form example explains the geometry of
Figures~\ref{fig:hero}~and~\ref{fig:obs-vs-guar}.  At $K = N/2-1$, the step test
probes a one-dimensional subspace that is invariant under the operator, while
admissibility is a uniform, coordinatewise property over the full
hypercube~$[0,1]^N$.  The constructive witness guaranteed by \Cref{thm:main}
makes the separation explicit: it aligns the input with the sign pattern
of~$h_K$ to prevent cancellations that occur in the step response.

\subsection{Mechanism in real space: kernel nonnegativity as the visible
  boundary}
\label{sec:real-space-mechanism}

\Cref{fig:kernel-triptych} provides a compact mechanism-level view.  At a
representative ``primary'' cutoff (we use $K = 511$ in the released figures), the
Fej\'er kernel is visibly nonnegative in real space, while the sharp and signed
kernels exhibit oscillatory sidelobes with negative troughs.  The insets quantify
these minima; together with the verified row-sum condition $\sum h_K = 1$, any
negative trough forces a corresponding positive mass elsewhere, which is
precisely what enables the $\{0,1\}$-valued witnesses achieving $y(0) < 0$
and~$y(0) > 1$ in \Cref{fig:hero}.

In other words, for this class of translation-invariant operators the
admissibility boundary is not a subtle functional-analytic phenomenon: it is
literally the sign of the impulse response.  The ``step looks fine'' regime is a
property of one particular probe signal, not a certificate of admissibility.

\section{Practical Certification: What to Check Instead of Step Tests}
\label{sec:certification}

The results above separate two notions that are often conflated in practice:
\begin{enumerate}[label=\arabic*.,leftmargin=2em]
  \item \emph{Signal-specific behavior} (e.g., whether a particular step
    response overshoots), and
  \item \emph{Operator-level admissibility} (uniform preservation of the unit
    interval for all inputs in~$[0,1]^N$).
\end{enumerate}

Step responses can be useful exploratory diagnostics, but they do not provide an
admissibility certificate.  In contrast, for the operator classes considered
here, admissibility admits a simple, complete check.

\subsection{Why a step test cannot certify admissibility}
\label{sec:why-not}

A step test probes a single input (or a small family of inputs).  Admissibility
is a uniform, worst-case property over a high-dimensional set.  The gap between
these is not merely quantitative; it is structural.

The Nyquist$-$1 example in \Cref{sec:nyquist-closed} exhibits the simplest
possible failure mode.  At $K = N/2-1$, the sharp and signed operators differ
from the identity only on the one-dimensional Nyquist subspace.  The canonical
midpoint step happens to be orthogonal to that subspace on even half-grids, hence
the step response is reproduced exactly.  This is not ``good behavior'' in an
operator-theoretic sense; it is an artifact of a particular probe signal missing
a particular oscillatory component.

More generally, any test that evaluates an operator on a finite collection of
inputs can be blind to directions that are not excited by those inputs.  In the
present setting, the mechanism is especially transparent: sharp truncation
produces an impulse response with alternating-sign tails (Dirichlet-type
ringing), and admissibility fails because those negative coefficients can be
isolated by a $\{0,1\}$-valued witness constructed from the kernel's sign pattern
(\Cref{thm:main}).  A step input, by contrast, mixes positive and negative lobes
and can exhibit substantial cancellation.  When cancellation aligns with a
missing Fourier mode (as at Nyquist$-$1), the step test becomes maximally
uninformative.

\subsection{A complete and constructive admissibility certificate}
\label{sec:complete-cert}

For any constant-preserving periodic convolution operator~$T_h$ (equivalently,
any circular filter with $\sum_n h[n] = 1$), admissibility is equivalent to the
pointwise nonnegativity of the impulse response~$h$.  This yields a direct
certificate:

\begin{itemize}[leftmargin=2em]
  \item Compute the impulse response~$h$ (e.g., by inverse FFT of the
    multiplier~$H$).
  \item Verify the mass condition $\sum_n h[n] = 1$ (or enforce it by
    construction).
  \item Check whether $\min_n h[n] \ge 0$.
\end{itemize}

If $\min_n h[n] \ge 0$, then~$T_h$ is admissible.  If $\min_n h[n] < 0$,
then~$T_h$ is not admissible, and moreover there is a certified worst-case
violation with an explicit witness input in~$\{0,1\}^N$.

A numerically stable way to report the failure magnitude is via the negative mass
$\Sminus(h) = \sum_{n:\,h[n]<0} h[n]$.  Under $\sum h = 1$, \Cref{thm:main}
implies a guaranteed bound
\[
  \exists\,x \in \{0,1\}^N:\quad
  \min(T_h x) \le \Sminus(h)
  \quad\text{and}\quad
  \max(T_h x) \ge 1 - \Sminus(h),
\]
so that the certified violation is at least~$-\Sminus(h)$ below~$0$ and at
least~$-\Sminus(h)$ above~$1$.  This bound is not heuristic: it is attained at a
specific output coordinate by the witness
$\xneg = \mathbf{1}_{\{h<0\}}$ (and similarly
$\xpos = \mathbf{1}_{\{h>0\}}$), applied to the corresponding row/shift of the
convolution.

Operationally, this means that an admissibility assessment can be reduced to a
single inverse FFT and a scan for negative entries---an $O(N\log N)$ procedure
that is both stronger and simpler than running suites of signal-specific tests.

\subsection{Design implication: ``spectral cutoff'' does not imply ``averaging
  operator''}
\label{sec:design}

In applications, sharp spectral truncation is often described informally as
``keeping low frequencies'' and therefore as a form of smoothing.
\Cref{sec:sweep} shows why this intuition is incomplete when the goal is
boundedness preservation: a low-pass multiplier can still correspond to an
oscillatory impulse response with negative lobes, and those negative lobes are
exactly what break admissibility.

The Fej\'er family illustrates the complementary point.  Although it is also a
low-pass construction, it corresponds to a nonnegative kernel and therefore to a
genuine averaging operator on~$[0,1]^N$.  In this sense, admissibility is not
primarily a statement about which frequencies are kept; it is a statement about
whether the real-space operator is a convex averaging map.

A practical corollary is that ``small observed overshoot on a step response''
should be treated as a quality metric for that probe signal, not as a safety
guarantee.  If admissibility matters (e.g., preserving physical bounds,
probabilities, or normalized intensities), then kernel nonnegativity is the
appropriate criterion, and the negative-mass witness provides a certified
fallback whenever the criterion fails.

\subsection{Recommended reporting standard}
\label{sec:reporting}

To make admissibility claims falsifiable and comparable across implementations,
we recommend that any report asserting ``boundedness preservation'' for a
translation-invariant filter include:

\begin{enumerate}[label=\arabic*.,leftmargin=2em]
  \item The grid size~$N$ and the multiplier definition~$H(k)$ (or
    kernel~$h[n]$).
  \item A mass-preservation check $\sum_n h[n] = 1$ to stated tolerance.
  \item The minimum coefficient $\min_n h[n]$.
  \item The negative mass~$\Sminus(h)$ (or equivalently the certified
    violation~$-\Sminus(h)$).
  \item If non-admissible, the explicit $\{0,1\}$-valued witness input that
    attains the bound at a specified index.
\end{enumerate}

These quantities are sufficient to reproduce the admissibility verdict and its
worst-case magnitude without reliance on probe-dependent behavior such as step
responses.

\section{Discussion}
\label{sec:discussion}

The central message of this work is modest in mathematical difficulty but
consequential in practice: signal-based tests are not certificates of
operator-level boundedness preservation.  In the periodic convolution setting,
admissibility is governed by a complete and checkable criterion---kernel
nonnegativity under unit mass---while step responses probe only a narrow slice of
the input space and can be systematically misleading.

\subsection{Why the distinction matters}

Boundedness preservation is not an aesthetic preference; it is often a modeling
requirement.  Examples include intensity, probability, or occupancy fields that
must remain in~$[0,1]$; concentrations or normalized densities that must remain
nonnegative and bounded; and monotone level-set or mask evolutions where values
outside~$[0,1]$ are physically meaningless or trigger downstream clipping.

In these settings, ``overshoot is small on a representative step'' is not a
safety statement.  It is compatible with an operator that admits inputs producing
large negative values and values well above one (\Cref{cor:violation}).  A
practitioner who relies on step plots as a proxy for boundedness risks building a
pipeline that behaves plausibly on common probes while failing sharply on
admissible but untested configurations.

\subsection{What is genuinely new here}

Several ingredients used in the proofs are classical (e.g., ``positive
coefficients yield convex combinations'').  The contribution is the operational
closure of the story in a setting where informal heuristics are common:

\begin{enumerate}[label=\arabic*.,leftmargin=2em]
  \item A short, explicit equivalence theorem (\Cref{thm:main}) tailored to
    admissibility on~$[0,1]^N$ for periodic convolution operators under a mass
    constraint.
  \item A constructive falsification mechanism that converts a single negative
    coefficient (or negative mass) into an explicit $\{0,1\}$-valued witness and
    a certified lower bound on the magnitude of failure.
  \item A demonstration of a practically important failure mode: near the
    Nyquist cutoff, the step response can be essentially exact while the
    guaranteed worst-case violation remains order one.
\end{enumerate}

Taken together, these points replace a folklore-style narrative (``low-pass
filters behave like averaging'') with a crisp boundary: admissibility is about
real-space positivity, not about discarding high frequencies.

\subsection{The ``blind spot'' phenomenon and why it is generic}

The near-Nyquist example is an extreme instance of a broader phenomenon: a probe
family can fail to excite the directions in which an operator is most
non-admissible.  In the present case, the mechanism is visible at the level of
kernels: sharp truncation produces oscillatory impulse responses with
alternating-sign sidelobes; a step input aggregates those sidelobes in a
structured way, and cancellations can be pronounced; at Nyquist$-$1 on even
grids, the canonical midpoint step happens to be orthogonal to the
one-dimensional Nyquist subspace where sharp truncation differs from the
identity, yielding an exact step output.

Nothing in this mechanism depends on fragile tuning.  It is a consequence of
using a small set of structured probes to test a uniform property over a large
domain.  The witness construction makes this mismatch concrete: it produces a
$\{0,1\}$-valued input aligned with the sign pattern of the negative kernel
lobes, eliminating the cancellations that steps can inadvertently exploit.

\subsection{Implications for filter design and for reporting practice}

If boundedness preservation is a requirement, then admissibility should be
treated as a design constraint rather than as a performance metric.  In practice
this leads to two recommendations.

First, when constructing spectral filters intended to act as averaging operators,
prefer constructions that are provably nonnegative in real space (e.g.,
Fej\'er-type kernels in one dimension, and their multidimensional analogues),
rather than sharp or signed truncations whose kernels necessarily oscillate.

Second, when reporting claims about boundedness preservation, provide an
operator-level certificate.  For periodic convolution, the minimal certificate
consists of the minimum kernel entry $\min_n h[n]$, the negative
mass~$\Sminus(h)$ (or equivalently a guaranteed violation magnitude), and, when
$\min h < 0$, the explicit witness input that attains the bound at a specified
coordinate.  These quantities are easy to compute and allow others to verify the
admissibility verdict without reproducing a particular set of probe plots.

\section{Related Work and Context}
\label{sec:related}

The tension between spectral truncation and real-space positivity is classical.
Sharp Fourier cutoffs are associated with Dirichlet kernels and with Gibbs-type
oscillations; by contrast, Ces\`aro/Fej\'er summation produces positive kernels
and improves convergence properties through averaging.  The present work draws on
this classical distinction but frames it in an operator-theoretic language
tailored to admissibility on the unit cube and to practical certification.

Positivity-preserving linear maps are fundamental across analysis and numerical
methods.  In finite dimensions, row-stochastic matrices with nonnegative entries
define Markov operators, preserving the simplex and the unit interval under
appropriate constraints.  \Cref{thm:main} can be read as a specialization of
this general viewpoint to circulant operators (periodic convolution), together
with a constructive witness that quantifies failure when negativity is present.

In numerical PDEs and signal processing, related ideas appear under maximum
principles, monotonicity constraints, and positivity-preserving discretizations.
The message here is aligned with that tradition: if one wants guaranteed bounds,
the correct object to certify is the operator, not a single probe trajectory.
The contribution is to make the certificate and the failure mechanism explicit for
commonly used spectral truncations, including a sharp ``false negative'' regime
where step-response testing is maximally misleading.

\section{Reproducibility, Artifacts, and Availability}
\label{sec:reproducibility}

\subsection{Overview}

All figures and numerical tables in this manuscript are generated by
deterministic scripts that emit a \emph{specification hash} (a SHA-256 digest of
the experiment specification) and a \emph{determinism hash} (a SHA-256 digest of
the resolved parameters, including the platform-independent configuration used to
produce the outputs).  These hashes are printed in the run logs and stored in
machine-readable JSON to support independent verification and to prevent ``silent
drift'' between revisions.

\subsection{Data and code availability}

The Zenodo record
(\href{https://doi.org/10.5281/zenodo.18689854}{doi:10.5281/zenodo.18689854})
archives the versioned software release, while the paper's figure-generation demo
and pre-generated figure artifacts are provided in the public repository at a
commit-pinned permalink (see the ``Data and Code Availability'' section preceding
the bibliography for exact paths and certified run identifiers).

The demonstration produces a self-contained JSON artifact containing the
numerical values underlying the figures (including the admissibility witnesses and
the certified violation bounds), along with the specification and determinism
hashes reported in the text.

\section{Conclusion}
\label{sec:conclusion}

This paper establishes a simple, complete, and practically useful
characterization of boundedness-preserving periodic convolution operators
on~$[0,1]^N$: under the standard mass constraint $\sum_n h[n] = 1$,
admissibility is equivalent to kernel nonnegativity.  The proof is short, but the
operational implications are substantial.

Two consequences are particularly relevant for common spectral filtering
workflows:

\begin{enumerate}[label=\arabic*.,leftmargin=2em]
  \item \textbf{Step responses are not certificates.}  Even when a step test
    shows no visible overshoot---indeed, even when the step is reproduced to
    machine precision---an operator may still be non-admissible.  The failure can
    be certified and large, as measured by the negative mass~$\Sminus(h)$ and
    realized by an explicit $\{0,1\}$-valued witness input.

  \item \textbf{Real-space positivity, not frequency cutoff, governs
    admissibility.}  Fej\'er-type averaging kernels remain admissible across
    cutoffs because they are nonnegative and unit-mass; sharp and signed
    truncations generally fail admissibility because their impulse responses
    necessarily oscillate and acquire negative lobes.
\end{enumerate}

In settings where values represent probabilities, intensities, normalized
concentrations, or any quantity requiring invariant bounds, admissibility should
be treated as a first-class constraint.  For periodic convolution, the
recommended practice is straightforward: compute $\min h$ and~$\Sminus(h)$.  If
either indicates negativity, admissibility fails and a concrete violating input
can be produced immediately.

\appendix

\section{Proofs and Quantitative Witnesses}
\label{app:proofs}

\subsection{Proof of Theorem~\ref{thm:main} (boundedness preservation
  $\Leftrightarrow$ nonnegative kernel)}
\label{app:proof-main}

We restate the theorem for convenience.

\begin{theoremrestate}[Discrete boundedness $\Leftrightarrow$ nonnegative coefficients;
  periodic convolution]
Let $h \in \R^N$ and define the periodic convolution operator
$T_h : \R^N \to \R^N$ by
\[
  (T_h x)[n]
  \;=\;
  \sum_{m=0}^{N-1} h[m]\; x[n-m]
  \qquad (n \;\mathrm{mod}\; N).
\]
Assume the mass constraint $\sum_{m=0}^{N-1} h[m] = 1$ (equivalently
$T_h \one = \one$).  Then the following are equivalent:
\begin{enumerate}[label=(\roman*)]
  \item $T_h$ maps~$[0,1]^N$ into~$[0,1]^N$.
  \item $h[m] \ge 0$ for all~$m$.
\end{enumerate}
\end{theoremrestate}

\begin{proof}
\textbf{(ii)~$\Rightarrow$~(i).}\;
Fix~$n$.  If $h[m] \ge 0$ for all~$m$ and $\sum_m h[m] = 1$, then
$(T_h x)[n]$ is a convex combination of entries of~$x$.  If
$x \in [0,1]^N$, then $0 \le x[\cdot] \le 1$, hence
\[
  0
  \;\le\;
  (T_h x)[n]
  = \sum_m h[m]\,x[n-m]
  \;\le\;
  \sum_m h[m] \cdot 1
  = 1.
\]
Since this holds for all~$n$, we have $T_h x \in [0,1]^N$.

\medskip\noindent
\textbf{(i)~$\Rightarrow$~(ii).}\;
Suppose, for contradiction, that $h[m_0] < 0$ for some~$m_0$.  Consider the
output at $n = 0$:
\[
  (T_h x)[0]
  \;=\;
  \sum_{m=0}^{N-1} h[m]\; x[-m].
\]
Define the binary witness input $\xneg \in \{0,1\}^N$ by
\[
  \xneg[-m]
  \;=\;
  \begin{cases}
    1, & h[m] < 0,\\
    0, & h[m] \ge 0.
  \end{cases}
\]
Then
\[
  (T_h \xneg)[0]
  \;=\;
  \sum_{m:\,h[m]<0} h[m]
  \;=:\;
  \Sminus(h).
\]
Since at least one term is strictly negative, $\Sminus(h) < 0$, so
$(T_h \xneg)[0] < 0$, contradicting~(i).

For the upper bound failure, define $\xpos \in \{0,1\}^N$ by
\[
  \xpos[-m]
  \;=\;
  \begin{cases}
    1, & h[m] > 0,\\
    0, & h[m] \le 0.
  \end{cases}
\]
Then
\[
  (T_h \xpos)[0]
  \;=\;
  \sum_{m:\,h[m]>0} h[m]
  \;=\;
  \sum_m h[m] - \sum_{m:\,h[m] \le 0} h[m].
\]
Under the mass constraint $\sum_m h[m] = 1$, and because
$\sum_{m:\,h[m] \le 0} h[m] \le \sum_{m:\,h[m]<0} h[m] = \Sminus(h)$, we obtain
\[
  (T_h \xpos)[0]
  \;\ge\;
  1 - \Sminus(h)
  = 1 + |\Sminus(h)|.
\]
Thus $(T_h \xpos)[0] > 1$, again contradicting~(i).  Therefore no negative
coefficient can exist, proving $h[m] \ge 0$ for all~$m$.
\end{proof}

\subsection{Certified worst-case violation bound}
\label{app:violation-bound}

The preceding proof yields a quantitative statement used throughout the
experiments.

\begin{corollaryrestate}[Guaranteed violation from negative mass]
Let $h \in \R^N$ satisfy $\sum_m h[m] = 1$.  Define
$\Sminus(h) = \sum_{m:\,h[m]<0} h[m] \le 0$.  If $\Sminus(h) < 0$, then there
exist binary inputs $\xneg, \xpos \in \{0,1\}^N$ such that
\[
  (T_h \xneg)[0] \;=\; \Sminus(h) \;<\; 0,
  \qquad
  (T_h \xpos)[0] \;\ge\; 1 - \Sminus(h) = 1 + |\Sminus(h)| \;>\; 1.
\]
In particular,
\[
  \min_n (T_h \xneg)[n] \;\le\; \Sminus(h),
  \qquad
  \max_n (T_h \xpos)[n] \;\ge\; 1 + |\Sminus(h)|.
\]
\end{corollaryrestate}

\begin{proof}
The constructions of~$\xneg$ and~$\xpos$ above are explicit, and the
equalities/inequalities follow by direct substitution and the mass constraint.
\end{proof}

\section{An Analytic Explanation of the Near-Nyquist ``Blind Spot'' for the Step}
\label{app:blind-spot}

This appendix explains why, on even grids, the step response of a sharp spectral
truncation can appear perfect near Nyquist while admissibility fails by a
certified amount.

\subsection{Sharp truncation at Nyquist$-$1 differs from the identity on a
  one-dimensional subspace}
\label{app:subspace}

Let~$N$ be even and consider the discrete Fourier transform (DFT) basis.  The
unique ``Nyquist mode'' corresponds to the frequency index $k = N/2$, whose
real-space representative is the alternating vector
\[
  v[n] \;=\; (-1)^n, \qquad n = 0, \dots, N-1.
\]
A sharp spectral truncation at cutoff $K = N/2-1$ retains every Fourier mode
except the Nyquist mode~$v$.  Equivalently, the associated linear
operator~$T_{\mathrm{sharp}}$ satisfies
\[
  T_{\mathrm{sharp}} \;=\; I - \Pi_{\mathrm{Nyq}},
\]
where~$\Pi_{\mathrm{Nyq}}$ denotes the orthogonal projection onto
$\spn\{v\}$ (with respect to the standard inner product on~$\R^N$
or~$\mathbb{C}^N$).  Thus, the operator differs from the identity only by
removing the Nyquist component.

\subsection{The midpoint step can be orthogonal to the Nyquist mode}
\label{app:orthogonal}

Consider the canonical periodic midpoint step
\[
  s[n] \;=\;
  \begin{cases}
    0, & 0 \le n \le N/2-1,\\
    1, & N/2 \le n \le N-1.
  \end{cases}
\]
Its Nyquist coefficient is proportional to the inner product
$\langle s, v \rangle$:
\[
  \langle s, v \rangle
  \;=\;
  \sum_{n=0}^{N-1} s[n]\,(-1)^n
  \;=\;
  \sum_{n=N/2}^{N-1} (-1)^n.
\]
If~$N/2$ is even (equivalently~$N$ is divisible by~$4$), then the sum on the
right contains an even number of alternating~$\pm 1$ terms and equals~$0$.  In
that common case,
\[
  \Pi_{\mathrm{Nyq}} s \;=\; 0,
  \qquad\text{hence}\qquad
  T_{\mathrm{sharp}}\, s \;=\; s.
\]
Therefore the step test reports exactly zero overshoot/undershoot at
$K = N/2-1$, despite the fact that the real-space impulse response remains
sign-changing and the operator is non-admissible.

This explains the ``blind spot'' seen in the near-Nyquist regime: the step probe
fails to excite the (one-dimensional) direction where the operator differs most
clearly from an averaging operator.

\subsection{Why admissibility can still fail dramatically}
\label{app:fail-dramatically}

The preceding calculation concerns a single probe~$s$.  Admissibility is a
uniform statement over~$[0,1]^N$.  If the kernel~$h$ of~$T_{\mathrm{sharp}}$ has
negative entries, then $\Sminus(h) < 0$ and \Cref{cor:violation} produces a
binary witness~$\xneg$ such that
$(T_{\mathrm{sharp}} \xneg)[0] = \Sminus(h) < 0$, independent of what happens
on the step.

In short: the step can be exactly preserved while the operator remains provably
unsafe on other admissible inputs.

\section{Reproducibility Checklist and Artifacts}
\label{app:repro}

\subsection{Experiment configuration}

The flagship demonstration and the figure-generation scripts use:

\begin{itemize}[leftmargin=2em]
  \item \textbf{Periodic grid size:} $N = 2048$.
  \item \textbf{Operators:} Fej\'er (Ces\`aro) averaging, sharp spectral
    truncation, and a signed control truncation, all normalized to satisfy the
    mass constraint $\sum_n h[n] = 1$.
  \item \textbf{Cutoff sets:} the coarse $K$-sweep and the dense sweep reported
    in the main text and depicted in Figures~\ref{fig:hero}--\ref{fig:obs-vs-guar}.
\end{itemize}

Each run emits a \emph{specification hash} (\texttt{spec\_sha256}) for the
experiment configuration and a \emph{determinism hash}
(\texttt{determinism\_sha256}) for the resolved parameters and outputs.  These
values are recorded in the run logs and in the generated JSON artifacts.

\medskip
\noindent
\textbf{Certified parameters for the flagship run:}
$(137,\, 107,\, 103)$ (primary triple),
$N = 2048$,
$K_{\mathrm{primary}} = 511$,
$\mathrm{Nyquist}-1 = 1023$.

\subsection{Artifact list}

The released archive and repository contain:

\begin{itemize}[leftmargin=2em]
  \item A machine-readable JSON file with the numerical arrays used to render
    Figures~\ref{fig:hero}--\ref{fig:obs-vs-guar} (including~$\Sminus(h)$,
    $\min h$, the step overshoot metric, and the guaranteed violation bounds
    across~$K$).
  \item The scripts required to regenerate the figures from scratch.
  \item The PNG figure artifacts referenced by the manuscript:
    \texttt{vendored/figures/Fig1\_HERO.png},
    \texttt{vendored/figures/Fig2\_KernelTriptych.png},
    \texttt{vendored/figures/Fig3\_ObservedVsGuaranteed.png}.
\end{itemize}

\section*{Data and Code Availability}
\addcontentsline{toc}{section}{Data and Code Availability}
\label{sec:data-avail}

An archival snapshot of the Marithmetics codebase is available via Zenodo record
18689854
(\href{https://doi.org/10.5281/zenodo.18689854}{doi:10.5281/zenodo.18689854}).
The exact figure-generation demo used for this manuscript, together with
pre-generated figure files and machine-readable run metadata, is available in the
public repository at commit
\texttt{346c9f3c8fe8f8a4e2a2f505b85d385b982baa28} under:
(i)~\texttt{codepack/demos/controllers/demo-11-step-tests-do-not-certify-%
admissibility/}
and
(ii)~\texttt{vendored/figures/}.
For convenience, commit-pinned permalinks are:

\smallskip
\noindent
\url{https://github.com/public-arch/Marithmetics/tree/346c9f3c8fe8f8a4e2a2f505b85d385b982baa28/codepack/demos/controllers/demo-11-step-tests-do-not-certify-admissibility}

\smallskip
\noindent
\url{https://github.com/public-arch/Marithmetics/tree/346c9f3c8fe8f8a4e2a2f505b85d385b982baa28/vendored/figures}

\smallskip
\noindent
The certified run identifiers for the flagship figures are:

\smallskip
\noindent
\texttt{spec\_sha256} =
\texttt{04c0fac520fef120bb7f9b46156046a7962be1be471ed16dba30060d790f15df}

\noindent
\texttt{determinism\_sha256} =
\texttt{3397465388c3d68082986d8f0204b8582a898021698a51e0e02dabeab3939fea}


\end{document}